\documentclass[]{interact}

\usepackage{epstopdf}
\usepackage[caption=false]{subfig}
\usepackage{xcolor}
\usepackage{amsmath}
\usepackage{amsmath,cases}
\usepackage{mathtools}

\usepackage[numbers,sort&compress]{natbib}
\bibpunct[, ]{[}{]}{,}{n}{,}{,}
\makeatletter
\def\NAT@def@citea{\def\@citea{\NAT@separator}}
\makeatother

\theoremstyle{plain}
\newtheorem{theorem}{Theorem}[section]
\newtheorem{lemma}[theorem]{Lemma}

\newtheorem{proposition}[theorem]{Proposition}

\theoremstyle{definition}
\newtheorem{definition}[theorem]{Definition}

\theoremstyle{remark}

\newcommand{\suppress}[1]{} 

\begin{document}

\articletype{Research Paper}

\title{\begin{center}Positive weak solutions of a double-phase variable exponent problem with a fractional-Hardy-type singular potential\\ and superlinear nonlinearity\end{center}}

\author{\name{Mustafa Avci\thanks{CONTACT M.~Avci. Email:  mavci@athabascau.ca (primary) \& avcixmustafa@gmail.com}}
\affil{Faculty of Science and Technology, Applied Mathematics, Athabasca University, AB, Canada}}

\maketitle

\begin{abstract}
In the present paper,  we study a double-phase variable exponent problem which is set up within a variational framework including a singular potential of fractional-Hardy-type. We employ the  Mountain-Pass theorem and the strong minimum principle to obtain the existence of at least one nontrivial positive weak solution.
\end{abstract}


\section{Introduction}
In this article, we study the following singular double-phase variable exponent problem

\begin{equation}\label{e1.1}
\begin{cases}
\begin{array}{rlll}
-\mathcal{D}(x,\mu,\nabla u)+\mathcal{S}(x,\mu,u)& =\lambda f(x,u) \text{ in }\Omega, \\
u & =0   \text{ on }\partial \Omega,\tag{$\mathcal{P}_{\lambda}$}
\end{array}
\end{cases}
\end{equation}
where
\begin{align*}
&\mathcal{D}(x,\mu,\nabla u):=\mathrm{div}(|\nabla u|^{p(x)-2}\nabla u+\mu(x)|\nabla u|^{q(x)-2}\nabla u), \\ &\mathcal{S}(x,\mu,u):=\frac{|u|^{p(x)-2}u}{|x|^{\alpha p(x)}}+\mu(x)\frac{|u|^{q(x)-2}u}{|x|^{\alpha q(x)}}.
\end{align*}
We assume that $\Omega$ is a bounded domain in $\mathbb{R}^N$ $(N\geq 3)$ with Lipschitz boundary $\partial \Omega$; $f$ is a Caratheodry function; $1<p,q \in C(\overline{\Omega })$; $0\leq \mu(\cdot)\in L^\infty(\Omega)$; $\lambda>0$ and $\alpha \in (0,1)$ are real parameters.\\

Problems involving the double phase operator have been extensively studied in the literature due to their significant applications in various fields. The double phase operator, defined as
$$
\mathrm{div}(|\nabla u|^{p(x)-2}\nabla u+\mu(x)|\nabla u|^{q(x)-2}\nabla u), \quad u \in W_{0}^{1,\mathcal{H}}(\Omega),
$$
is intrinsically connected to the two-phase integral functional
$$
u \to \int_{\Omega} \left(|\nabla u|^{p(x)} + \mu(x)|\nabla u|^{q(x)}\right) \mathrm{d}x, \quad u \in W_{0}^{1,\mathcal{H}}(\Omega),
$$
where $W_{0}^{1,\mathcal{H}}(\Omega)$ denotes a Musielak-Orlicz Sobolev space. This functional exhibits changing ellipticity properties depending on the set where the weight function $\mu(\cdot)$ vanishes, leading to two distinct phases of elliptic behavior. The pioneering work of Zhikov \cite{Zhikov1} first investigated this functional with constant exponents to model strongly anisotropic materials. In elasticity theory, $\mu(\cdot)$ encodes geometric information about composites consisting of two different materials with power hardening exponents $p(\cdot)$ and $q(\cdot)$ \cite{Zhikov2}. Early mathematical treatments of these two-phase integrals were conducted by Baroni-Colombo-Mingione \cite{Baroni1,Baroni2,Baroni3}, Colombo-Mingione \cite{Colombo1,Colombo2}, and Ragusa-Tachikawa \cite{Ragusa}, with further contributions from De Filippis-Mingione \cite{Filippis} for nonautonomous integrals.

While numerous authors have established existence and multiplicity results for double phase problems with constant exponents \cite{Biagi,Colasuonno,Gasinski1,Gasinski2,Liuw,Papageorgiou,Perera,Zeng2}, the variable exponent case remains comparatively less explored. Recent contributions in this direction include works by Amoroso-Bonanno-D’Aguì-Winkert \cite{Amoroso}, Abergi-Bennouna-Benslimane-Ragusa \cite{Aberqi}, Bahrouni-Radulescu-Winkert \cite{Bahrouni}, Cen-Kim-Kim-Zeng \cite{Cen}, Crespo-Blanco-Gasinski-Harjulehto-Winkert \cite{Crespo-Blanco}, Leonardi-Papageorgiou \cite{Leonardi}, Liu-Pucci \cite{Liu}, Kim-Kim-Oh-Zeng \cite{Kim}, Vetro-Winkert \cite{Vetro}, Zeng-Radulescu-Winkert \cite{Zeng1}, and Avci \cite{Avci3}.

The distinctive feature of problem \eqref{e1.1} is the presence of the singular Hardy-type potential in the term $\mathcal{B}(x,\mu,u)$, which introduces critical singularities at the origin. This represents a significant departure from the existing literature on double phase problems. Physically, such singular potentials model phenomena involving strongly attractive forces or singular sources, appearing in quantum mechanics, molecular physics, and relativistic field theories. In material science, they can describe composites with point defects or singular material properties that intensify near specific locations. Mathematically, the fractional-type-Hardy potential introduces substantial challenges due to its critical nature and the lack of compactness in the associated energy functional which requires a careful treatment of the underlying Musielak-Orlicz Sobolev spaces.

Our work aims to bridge the gap between double phase operators and singular potentials in the variable exponent setting, establishing new existence results that extend the current understanding of these complex nonlinear phenomena.

\section{Mathematical Background and Preliminaries}

We start with some basic concepts of variable Lebesgue-Sobolev spaces. For more details, and the proof of the following propositions, we refer the reader to \cite{Cruz,Diening,Fan,Radulescu}.\\
Define
\begin{equation*}
C_{+}\left( \overline{\Omega }\right) =\left\{h\in C\left( \overline{\Omega }\right): h(x)>1 \text{ for all\ }x\in
\overline{\Omega }\right\} .
\end{equation*}
For $h\in C_{+}( \overline{\Omega }) $ denote
\begin{equation*}
h^{-}:=\underset{x\in \overline{\Omega }}{\min }h( x), \quad  h^{+}:=\underset{x\in \overline{\Omega }}{\max}h(x) <\infty .
\end{equation*}
For a given $h\in C_{+}\left( \overline{\Omega }\right) $, we define \textit{the
variable exponent Lebesgue space} by
\begin{equation*}
L^{h(x)}(\Omega) =\left\{ u\mid u:\Omega\rightarrow\mathbb{R}\text{ is measurable},\int_{\Omega }|u(x)|^{h(x) }dx<\infty \right\},
\end{equation*}
equipped with the Luxemburg norm given by
\begin{equation*}
|u|_{h(x)}=\inf \left\{ \lambda>0:\int_{\Omega }\left\vert \frac{u(x)}{\lambda }\right\vert ^{h(x)}dx\leq 1\right\}.
\end{equation*}
With this norm, $(L^{h(x)}(\Omega),|\cdot|_{h(x)})$ becomes a separable and reflexive Banach space.
\begin{proposition}[H\"{o}lder's Inequality]\label{Prop:2.1} For any $u\in L^{h( x) }(\Omega)$ and $v\in L^{h^{\prime}(x)}(\Omega)$, we have
\begin{equation*}
\int_{\Omega}|uv|dx\leq C(h^{-},(h^{-})^{\prime})|u|_{h(x)}|v|_{h^{\prime}(x)},
\end{equation*}
where $L^{h^{\prime }(x) }(\Omega) $ is conjugate space of $L^{h( x) }(\Omega)$ such that $\frac{1}{h( x) }+\frac{1}{h^{\prime}( x)}=1$.
\end{proposition}

The convex functional $\rho :L^{h(x) }(\Omega) \rightarrow\mathbb{R}$ defined by
\begin{equation*}
\rho(u) =\int_{\Omega }|u(x)|^{h(x)}dx,
\end{equation*}
is called modular on $L^{h(x) }(\Omega)$.

\begin{proposition}\label{Prop:2.2} If $u,u_{n}\in L^{h(x) }(\Omega)$, we have
\begin{itemize}
\item[$(i)$] $|u|_{h(x) }<1 ( =1;>1) \Leftrightarrow \rho(u) <1 (=1;>1);$
\item[$(ii)$] $|u|_{h( x) }>1 \implies |u|_{h(x)}^{h^{-}}\leq \rho(u) \leq |u|_{h( x) }^{h^{+}}$;\newline
$|u|_{h(x) }\leq1 \implies |u|_{h(x) }^{h^{+}}\leq \rho(u) \leq |u|_{h(x) }^{h^{-}};$
\item[$(iii)$] $\lim\limits_{n\rightarrow \infty }|u_{n}|_{h(x)}=0\Leftrightarrow \lim\limits_{n\rightarrow \infty }\rho (u_{n})=0; \newline
\lim\limits_{n\rightarrow \infty }|u_{n}|_{h(x)}=\infty \Leftrightarrow \lim\limits_{n\rightarrow \infty }\rho(u_{n})=\infty$.
\end{itemize}
\end{proposition}
\begin{proposition}\label{Prop:2.3aa} If $u,u_{n}\in L^{h(x)}(\Omega)$, then the following statements are
equivalent:
\begin{itemize}
\item[$(i)$] $\lim\limits_{n\rightarrow \infty}|u_{n}-u|_{h(x)}=0$;
\item[$(ii)$] $\lim\limits_{n\rightarrow \infty }\rho(u_{n}-u)=0$;
\item[$(iii)$] $u_{n}\rightarrow u\mathit{\ }$\textit{in measure in}$\mathit{\ }\Omega \mathit{\ }$\textit{and}
$\mathit{\ }\lim\limits_{n\rightarrow \infty}\rho (u_{n})=\rho(u)$.
\end{itemize}
\end{proposition}
The variable exponent Sobolev space $W^{1,h(x)}( \Omega)$ is defined by
\begin{equation*}
W^{1,h(x)}( \Omega) =\{u\in L^{h(x) }(\Omega) : |\nabla u| \in L^{h(x)}(\Omega)\},
\end{equation*}
with the norm
\begin{equation*}
\|u\|_{1,h(x)}=|u|_{h(x)}+|\nabla u|_{h(x)},
\end{equation*}
or equivalently
\begin{equation*}
\|u\|_{1,h(x)}=\inf \left\{ \lambda>0:\int_{\Omega }\left( \left|\frac{\nabla u(x) }{\lambda}\right|^{h(x)}+\left|\frac{u(x)}{\lambda}\right|^{h(x)}\right)dx
\leq 1\right\},
\end{equation*}
for all $u\in W^{1,h(x)}(\Omega)$, where $|\nabla u|_{h(x)}=|\,|\nabla u|\,|_{h(x)}$.\\

The space $W_{0}^{1,h(x)}(\Omega)$ is defined as the closure of $C_{0}^{\infty }(\Omega )$ in $W^{1,h(x)}(\Omega)$.
\begin{proposition}\label{Prop:2.4} If $h\in C_+(\overline{\Omega })$ and $ h^{+}<\infty$, then the spaces $L^{h(x) }( \Omega)$, $W^{1,h(x)}(\Omega)$, and $W_{0}^{1,h(x)}(\Omega)$ are separable and reflexive Banach spaces.
\end{proposition}

Furthermore,  Poincar\'{e} inequality holds in $W_{0}^{1,h(x)}(\Omega)$ \cite{Fan}; that is, there exists a
positive constant $c$ independent of $u$ such that
\begin{equation*}
|u|_{h(x)}\leq c|\nabla u|_{h(x)},\quad \forall u\in W_{0}^{1,h(x)}(\Omega),
\end{equation*}
which implies that $|\nabla u|_{h(x)}$ is an equivalent norm in $W_{0}^{1,h(x)}(\Omega)$. Therefore, on $W_{0}^{1,h(x)}(\Omega)$ we can define an equivalent norm $\|\cdot\|_{h(x)}$ such that
\begin{equation*}
\|u\|_{h(x)} =|\nabla u|_{h(x)}.
\end{equation*}

\begin{proposition}\label{Prop:2.5} If $h\in C_+(\overline{\Omega })$, $r\in C(\overline{\Omega })$ and  $1\leq r(x)<h^{\ast }(x)$ for all $x\in
\overline{\Omega }$, then the embeddings $W^{1,h(x)}(\Omega) \hookrightarrow L^{r(x) }(\Omega)$ and $W_0^{1,h(x)}(\Omega) \hookrightarrow L^{r(x) }(\Omega)$  are compact and continuous, where

$h^{\ast }( x) =\left\{\begin{array}{cc}
\frac{Nh(x) }{N-h( x) } & \text{if }h(x)<N, \\
+\infty & \text{if }h( x) \geq N.
\end{array}
\right. $
\end{proposition}

\medskip
In the sequel, we introduce the double phase operator, the Musielak–Orlicz space, and the Musielak–Orlicz Sobolev space, respectively.\\
\medskip

Throughout the paper, we assume the following.
\begin{itemize}
\item[$(H_1)$] $p,q \in C_+(\overline{\Omega})$, $1<p(x)<N$ and $p(x)<q(x)$ for all $x\in \overline{\Omega }$ with \newline $q^+<\min\{p^*(x),N\}$.
\item[$(H_2)$] $\mu\in L^\infty(\Omega)$ such that $\mu(\cdot)\geq 0$, and $\mu(\cdot)\not\equiv 0$.
\end{itemize}

Let $\mathcal{H}:\Omega\times [0,\infty]\to [0,\infty]$ be the nonlinear function defined by
\[
\mathcal{H}(x,t)=t^{p(x)}+\mu(x)t^{q(x)}\ \text{for all}\ (x,t)\in \Omega\times [0,\infty).
\]
Then the corresponding modular $\rho_\mathcal{H}(\cdot)$ is given by
\[
\displaystyle\rho_\mathcal{H}(u)=\int_\Omega\mathcal{H}(x,|u|)dx=
\int_\Omega\left(|u|^{p(x)}+\mu(x)|u|^{q(x)}\right)dx.
\]
The \textit{Musielak-Orlicz space} $L^{\mathcal{H}}(\Omega)$, is defined by
\[
L^{\mathcal{H}}(\Omega)=\left\{u:\Omega\to \mathbb{R}\,\, \text{measurable};\,\, \rho_{\mathcal{H}}(u)<+\infty\right\},
\]
endowed with the Luxemburg norm
\[
\|u\|_{\mathcal{H}}=\inf\left\{\hat{\lambda}>0: \rho_{\mathcal{H}}\left(\frac{u}{\hat{\lambda}}\right)\leq 1\right\}.
\]
Analogous to Proposition \ref{Prop:2.2}, there are similar relationship between the modular $\rho_{\mathcal{H}}(\cdot)$ and the norm $\|\cdot\|_{\mathcal{H}}$, see \cite[Proposition 2.13]{Crespo-Blanco} for a detailed proof.

\begin{proposition}\label{Prop:2.2a}
Assume $(H_1)$ hold, $u\in L^{\mathcal{H}}(\Omega)$ and $\hat{\lambda}\in \mathbb{R}$. Then
\begin{itemize}
\item[$(i)$] If $u\neq 0$, then $\|u\|_{\mathcal{H}}=\hat{\lambda}\Leftrightarrow \rho_{\mathcal{H}}(\frac{u}{\hat{\lambda}})=1$,
\item[$(ii)$] $\|u\|_{\mathcal{H}}<1\ (\text{resp.}\ >1, =1)\Leftrightarrow \rho_{\mathcal{H}}(\frac{u}{\hat{\lambda}})<1\ (\text{resp.}\ >1, =1)$,
\item[$(iii)$] If $\|u\|_{\mathcal{H}}<1\Rightarrow \|u\|_{\mathcal{H}}^{q^+}\leq \rho_{\mathcal{H}}(u)\leq \|u\|_{\mathcal{H}}^{p^-}$,
\item[$(iv)$]If $\|u\|_{\mathcal{H}}>1\Rightarrow \|u\|_{\mathcal{H}}^{p^-}\leq \rho_{\mathcal{H}}(u)\leq \|u\|_{\mathcal{H}}^{q^+}$,
\item[$(v)$] $\|u\|_{\mathcal{H}}\to 0\Leftrightarrow \rho_{\mathcal{H}}(u)\to 0$,
\item[$(vi)$]$\|u\|_{\mathcal{H}}\to +\infty\Leftrightarrow \rho_{\mathcal{H}}(u)\to +\infty$,
\item[$(vii)$] $\|u\|_{\mathcal{H}}\to 1\Leftrightarrow \rho_{\mathcal{H}}(u)\to 1$,
\item[$(viii)$] If $u_n\to u$ in $L^{\mathcal{H}}(\Omega)$, then $\rho_{\mathcal{H}}(u_n)\to\rho_{\mathcal{H}}(u)$.
\end{itemize}
\end{proposition}

\medskip
\noindent
The \textit{Musielak-Orlicz Sobolev space} $W^{1,\mathcal{H}}(\Omega)$ is defined by
\[
W^{1,\mathcal{H}}(\Omega)=\left\{u\in L^{\mathcal{H}}(\Omega):
|\nabla u|\in L^{\mathcal{H}}(\Omega)\right\},
\]
and equipped with the norm
\[
\|u\|_{1,\mathcal{H}}=\|\nabla u\|_{\mathcal{H}}+\|u\|_{\mathcal{H}},
\]
where $\|\nabla u\|_{\mathcal{H}}=\|\,|\nabla u|\,\|_{\mathcal{H}}$.\\

\medskip
\noindent
The space $W_0^{1,\mathcal{H}}(\Omega)$ is defined as the closure of $C_{0}^{\infty }(\Omega )$ in $W^{1,\mathcal{H})}(\Omega)$.
Note also that $L^{\mathcal{H}}(\Omega), W^{1,\mathcal{H}}(\Omega)$ and $W_0^{1,\mathcal{H}}(\Omega)$ are reflexive Banach spaces \cite[Proposition 2.12]{Crespo-Blanco}.\\

\medskip
\noindent
We now present the following embedding relations given in \cite[Proposition 2.16]{Crespo-Blanco}.
\begin{proposition}\label{Prop:2.7a}
Assume that $(H_1)$ and $(H_2)$ hold. Then the following embeddings hold:
\begin{itemize}
\item[$(i)$] $L^{\mathcal{H}}(\Omega)\hookrightarrow L^{h(x)}(\Omega), W^{1,\mathcal{H}}(\Omega)\hookrightarrow W^{1,h(x)}(\Omega)$, $W_0^{1,\mathcal{H}}(\Omega)\hookrightarrow W_0^{1,h(x)}(\Omega)$ are continuous for $h\in C(\overline{\Omega})$ with $1\leq h(x)\leq p(x)$ for all $x\in \overline{\Omega}$.
\item[$(ii)$] $W^{1,\mathcal{H}}(\Omega)\hookrightarrow L^{h(x)}(\Omega)$ and $W_0^{1,\mathcal{H}}(\Omega)\hookrightarrow L^{h(x)}(\Omega)$ are compact for $h\in C(\overline{\Omega})$ with $1\leq h(x)< p^*(x)$ for all $x\in \overline{\Omega}$.
\end{itemize}
\end{proposition}

\begin{proposition}\label{Prop:2.7c}\cite{Crespo-Blanco}
Assume that $(H_1)$ and $(H_2)$ hold. Then the following hold:
\begin{itemize}
\item[$(i)$] The embedding $W^{1,\mathcal{H}}(\Omega)\hookrightarrow L^{\mathcal{H}}(\Omega)$ is compact;
\item[$(ii)$] There exists a constant $c>0$ independent of $u$ such that
\[
\|u\|_{\mathcal{H}}\leq c\|\nabla u\|_{\mathcal{H}}, \quad \forall u \in W_0^{1,\mathcal{H}}(\Omega).
\]
\end{itemize}
\end{proposition}

\medskip
\noindent
As a conclusion of Proposition \ref{Prop:2.7c}, the space $W_0^{1,\mathcal{H}}(\Omega)$ can be equipped with an equivalent norm $\|\cdot\|_{1,\mathcal{H},0}$ given by
\[
\|u\|_{1,\mathcal{H},0}=\|\nabla u\|_{\mathcal{H}}, \quad \forall u \in W_0^{1,\mathcal{H}}(\Omega).
\]

\begin{proposition}\label{Prop:2.7}
The convex functional
$$
\varrho_{\mathcal{H}}(u):=\int_\Omega\left(\frac{|\nabla u|^{p(x)}}{p(x)}+\mu(x)\frac{|\nabla u|^{q(x)}}{q(x)}\right)dx
$$
is of class $ C^{1}(W_0^{1,\mathcal{H}}(\Omega), \mathbb{R})$, and its derivative $\varrho^{\prime}_{\mathcal{H}}$ satisfies the ($S_+$)-property \cite{Zeidler} with the derivative
$$
\langle\varrho^{\prime}_{\mathcal{H}}(u),\varphi\rangle=\int_{\Omega}(|\nabla u|^{p(x)-2}\nabla u+\mu(x)|\nabla u|^{q(x)-2}\nabla u)\cdot\nabla \varphi dx,
$$
for all  $u, \varphi \in W_0^{1,\mathcal{H}}(\Omega)$, where $\langle \cdot, \cdot\rangle$ is the dual pairing between $W_0^{1,\mathcal{H}}(\Omega)$ and its dual $W_0^{1,\mathcal{H}}(\Omega)^{*}$ \cite{Crespo-Blanco}.
\end{proposition}

Next, we present and prove a fractional-Hardy-type inequality which we use to keep the singular integral $\int_{\Omega}\left(\frac{|u|^{p(x)}}{|x|^{\alpha p(x)}}+\mu(x)\frac{|u|^{q(x)}}{|x|^{\alpha q(x)}}\right)dx$ under the control so that it does not exhibit an uncontrolled growth (blow up). Note that when $\alpha\equiv 1$, the multi-phase case of \eqref{e3.11mm} is given in \cite{Avci1}, while the anisotropic case is given in \cite{Avci2}. Also we refer the reader to \cite{Fiscella} for constant exponent case.

\begin{lemma}\label{Lem:3.3abb}
Assume that $(H_1)$ and $(H_2)$ hold, and let $\alpha\in(0,1)$. Then the following assertions hold.
\begin{enumerate}
\item[$(i)$] If, additionally,
\begin{equation}\label{eq:Lem33abb:balance}
q^+\leq r_\alpha:=\frac{Np^-}{N-(1-\alpha)p^-},
\end{equation}
then there exists a constant $\hat H>0$, independent of $u$, such that
\begin{equation}\label{e3.11mm}
\int_\Omega\left(\frac{|u|^{p(x)}}{p(x)|x|^{\alpha p(x)}}+\mu(x)\frac{|u|^{q(x)}}{q(x)|x|^{\alpha q(x)}}\right)dx\leq \hat H\|u\|_{1,\mathcal{H},0}^{\kappa_+}
\end{equation}
for every $u\in W_0^{1,\mathcal H}(\Omega)$, where $\|u\|_{1,\mathcal{H},0}^{\kappa_+}=\max\{\|u\|_{1,\mathcal{H},0}^{p^-},\|u\|_{1,\mathcal{H},0}^{q^+}\}$.\\
\item[$(ii)$] Choose $M>\max\{1,\sup_{x\in\Omega}|x|\}$.
Then, without requiring \eqref{eq:Lem33abb:balance},
\begin{equation}\label{e3.11mn}
\int_\Omega\left(\frac{|u|^{p(x)}}{|x|^{\alpha p(x)}}+\mu(x)\frac{|u|^{q(x)}}{|x|^{\alpha q(x)}}\right)dx
\geq \frac{1}{M^{\alpha q^+}}\|u\|_{\mathcal{H}}^{\kappa_-}
\end{equation}
for every $u\in W_0^{1,\mathcal H}(\Omega)$, where $\|u\|_{\mathcal{H}}^{\kappa_-}=\min\{\|u\|_{\mathcal{H}}^{p^-},\|u\|_{\mathcal{H}}^{q^+}\}$.
\end{enumerate}
\end{lemma}

\begin{proof}
\textit{$(i)$} First, we establish some ground work.  Since $1<p^-<N$ and $L^{\mathcal H}(\Omega)$ embeds continuously into $L^{p^{-}}(\Omega)$, we have $W_0^{1,\mathcal H}(\Omega)\hookrightarrow W_0^{1,p^-}(\Omega)$. For $w\in W_0^{1,p^{-}}(\Omega)$, if we extend $w$ by zero to $\mathbb R^N$, then the classical (constant exponent) Hardy, and Sobolev inequalities give
\begin{equation}\label{e3.11mnk}
\left\|\frac{w}{|x|}\right\|_{L^{p^{-}}(\Omega)}\leq \frac{p^-}{N-p^-}\|\nabla w\|_{L^{p^{-}}(\Omega)},\quad \text{and} \quad \|w\|_{L^{(p^{-})^*}(\Omega)}\leq S_{p^-}\|\nabla w\|_{L^{p^{-}}(\Omega)},
\end{equation}
respectively, where $S_{p^-}$ is the Sobolev constant.\\ Next, applying H\"older's inequality to $\left\|\frac{w}{|x|^\alpha}\right\|_{L^{r_\alpha}(\Omega)}^{r_\alpha}$ with $
\frac1{r_\alpha}=\frac{\alpha}{p^-}+\frac{1-\alpha}{(p^-)^*}$ leads to
\begin{align}\label{eq:Lem33abb:weighted_embedding-0}
&\int_{\Omega} \left(\frac{|w|^{\alpha}|w|^{1-\alpha}}{|x|^\alpha}\right)^{r_\alpha}dx\nonumber \\
&\leq \left(\int_{\Omega} \left(\frac{|w|^{\alpha}}{|x|^\alpha}\right)^{r_\alpha \frac{p^-}{\alpha r_\alpha}}dx\right)^{\frac{\alpha r_\alpha}{p^-}}\left(\int_{\Omega} \left(|w|^{1-\alpha}\right)^{r_\alpha \frac{(p^-)^*}{(1-\alpha)r_\alpha}}dx\right)^{\frac{(1-\alpha)r_\alpha}{(p^-)^*}},
\end{align}
and then using \eqref{e3.11mnk} yields
\begin{align}\label{eq:Lem33abb:weighted_embedding}
\left\|\frac{w}{|x|^\alpha}\right\|_{L^{r_\alpha}(\Omega)}
\leq \left\|\frac{w}{|x|}\right\|_{L^{p^-}(\Omega)}^\alpha\|w\|_{L^{(p^-)^*}(\Omega)}^{1-\alpha}
\leq C_\alpha\|\nabla w\|_{L^{p^-}(\Omega)},
\end{align}
where $C_\alpha:=\left(\frac{p^-}{N-p^-}\right)^\alpha S_{p^-}^{1-\alpha}$.\\
Now, we can deal with \eqref{e3.11mm}: if $u=0$, the conclusion is immediate; thus we assume that $u\neq0$.
Let
\[
v:=\frac{u}{t},\quad t:=\|u\|_{1,\mathcal H,0}=\|\nabla u\|_{\mathcal H}>0.
\]
The norm-modular unit ball property (Proposition \ref{Prop:2.2a}) then gives $\rho_{\mathcal H}(\nabla v)=1$. Therefore,
\begin{equation}\label{e3.11mnkl}
\int_\Omega|\nabla v|^{p^-} dx \leq C_{\{p^+,|\Omega|\}}+\int_\Omega|\nabla v|^{p(x)} dx \leq C_{\{p^+,|\Omega|\}}+ \rho_{\mathcal H}(\nabla v)\leq C_{\{p^+,|\Omega|\}}+1,
\end{equation}
where $C_{\{p^+,|\Omega|\}}:=2^{p^+-1}|\Omega|$. Consequently, \eqref{eq:Lem33abb:weighted_embedding} implies
\begin{equation}\label{e3.11mnkm}
\left\|\frac{v}{|x|^\alpha}\right\|_{L^{r_\alpha}(\Omega)}\leq C_\alpha(C_{\{p^+,|\Omega|\}}+1)^{1/p^-}.
\end{equation}

First note that $\left(\frac{|v|}{|x|^\alpha}\right)^\sigma\leq 1 + \left(\frac{|v|}{|x|^\alpha}\right)^{r_\alpha}$ holds for $0<\sigma\leq r_\alpha$. By $(H_1)$ and \eqref{eq:Lem33abb:balance}, we have $p(x)<q(x)\leq q^+\leq r_\alpha$, and hence
\begin{align}\label{e3.11mnka}
K_\alpha(v)&:=\int_\Omega\left(\frac{|v|^{p(x)}}{|x|^{\alpha p(x)}}+\mu(x)\frac{|v|^{q(x)}}{|x|^{\alpha q(x)}}
\right) dx\nonumber\\
& \leq (1+\|\mu\|_\infty)\left(|\Omega|+\int_\Omega \left(\frac{|v|}{|x|^\alpha}\right)^{r_\alpha}dx\right)\leq C_0,
\end{align}
where one may choose
\[
C_0:=(1+\|\mu\|_\infty)\left[|\Omega|+\left(C_\alpha(C_{\{p^+,|\Omega|\}}+1)^{1/p^-}\right)^{r_\alpha}\right].
\]
Since every exponent $p(x)$ and $q(x)$ lies in $[p^-,q^+]$ for every $x\in\Omega$, we  have
\[
t^{p(x)},\,t^{q(x)}\leq\max\{t^{p^-},t^{q^+}\}.
\]
Thus, substituting $v$  back  gives
\[
K_\alpha(u)\leq C_0\max\{\|u\|_{1,\mathcal H,0}^{p^-},\|u\|_{1,\mathcal H,0}^{q^+}\}.
\]
Finally, since $1/p(x)\leq 1/p^-$ and $1/q(x)\leq 1/p^-$, \eqref{e3.11mm} follows with $\hat H:=C_0/p^-$.

\medskip
\noindent
\textit{$(ii)$} For almost every $x\in\Omega$, it holds
\[
|x|^{\alpha p(x)}<M^{\alpha q^+},\quad |x|^{\alpha q(x)}<M^{\alpha q^+},
\]
which yields
\begin{align}\label{e3.11mnkag}
 K_\alpha(u)\ge M^{-\alpha q^+}\rho_{\mathcal H}(u).
\end{align}
Note that since the $p(x)$-term is positive on a set of positive measure, the inequality \eqref{e3.11mnkag} is strict when $u\neq 0$. Therefore, if $s:=\|u\|_{\mathcal H}>0$, then $\rho_{\mathcal H}(u/s)=1$ and
\begin{align*}
\rho_{\mathcal H}(u)&=\int_\Omega\left(s^{p(x)}\left|\frac{u}{s}\right|^{p(x)}+\mu(x)s^{q(x)}\left|\frac{u}{s}\right|^{q(x)}\right)dx\\
&\geq\min\{s^{p^-},s^{q^+}\}\rho_{\mathcal H}(u/s)=\min\{s^{p^-},s^{q^+}\}
\end{align*}
which proves \eqref{e3.11mn}.
\end{proof}

\section{Variational Framework and Main Results}
We define the energy functional $\mathcal{I}:W_0^{1,\mathcal{H}}(\Omega)\rightarrow \mathbb{R}$ corresponding to equation (\ref{e1.1}) by
\begin{align*}
\mathcal{I}(u)=
&\int_{\Omega}\left(\frac{|\nabla u|^{p(x)}}{p(x)}+\mu(x)\frac{|\nabla u|^{q(x)}}{q(x)}\right)dx+\int_{\Omega}\left(\frac{|u|^{p(x)}}{p(x)|x|^{\alpha p(x)}}+\mu(x)\frac{|u|^{q(x)}}{q(x)|x|^{\alpha q(x)}}\right)dx\\
&-\lambda\int_{\Omega}F(x,u)dx,
\end{align*}
or
\begin{align*}
\mathcal{I}(u)=\varrho_{\mathcal{H}}(u)+\mathcal{G}(u)-\lambda\int_{\Omega}F(x,u)dx,
\end{align*}
where $\mathcal{G}(u):=\int_{\Omega}\left(\frac{|u|^{p(x)}}{p(x)|x|^{\alpha p(x)}}+\mu(x)\frac{|u|^{q(x)}}{q(x)|x|^{\alpha q(x)}}\right)dx $.

\begin{definition}\label{Def:3.1} A function $u\in W_0^{1,\mathcal{H}}(\Omega)$ is called a (weak) solution to problem \eqref{e1.1} if
\begin{align}\label{e3.2}
&\int_{\Omega}(|\nabla u|^{p(x)-2}\nabla u+\mu(x)|\nabla u|^{q(x)-2}\nabla u)\cdot\nabla \varphi dx+\int_{\Omega}\left(\frac{|u|^{p(x)-2}u}{|x|^{\alpha p(x)}}+\mu(x)\frac{|u|^{q(x)-2}u}{|x|^{\alpha q(x)}}\right)\varphi dx\nonumber\\
&=\lambda\int_{\Omega}f(x,u) \varphi dx,\,\,\ \forall \varphi\in W_0^{1,\mathcal{H}}(\Omega),
\end{align}
where $F(x,t)=\int_{0}^{t}f(x,s)ds$.
\end{definition}

\begin{lemma}\label{Lem:3.1}
$\mathcal{I} \in C^{1}(W_0^{1,\mathcal{H}}(\Omega),\mathbb{R})$ with the derivative
\begin{align}\label{e3.2aa}
\langle \mathcal{I}^\prime (u), \varphi \rangle&=\langle\varrho^{\prime}_{\mathcal{H}}(u),\varphi\rangle+\langle\mathcal{G}^{\prime}(u),\varphi\rangle-\lambda\langle f(x,u), \varphi\rangle,\,\,\ \forall \varphi\in W_0^{1,\mathcal{H}}(\Omega)
\end{align}
Moreover, the critical points of $\mathcal{I}$ are the solutions of problem \eqref{e1.1}.
\end{lemma}
\begin{proof}
By Proposition \ref{Prop:2.7}, $\varrho_{\mathcal{H}}$  is a continuously G\^{a}teaux differentiable functional with the derivative $\langle\varrho^{\prime}_{\mathcal{H}}(\cdot),\varphi\rangle$. Moreover, using the assumption $(\mathbf{f}_{1})$, and applying Holder inequality along with the related embeddings shows that the functional $\int_{\Omega}F(x,u)dx$  is also continuously G\^{a}teaux differentiable with the derivative $\langle f(x,u), \varphi\rangle=\int_{\Omega}f(x,u) \varphi dx$. Thus, it remains only to determine the G\^{a}teaux derivative of $\mathcal{G}$ and show that it has the same regularity on $W_0^{1,\mathcal{H}}(\Omega)$.

To begin, by the Mean Value Theorem, there are  $0< \varepsilon_1, \varepsilon_2< 1$ such that
\begin{align}\label{e3.21d}
\langle \mathcal{G}^{\prime}(u),\varphi \rangle&=\lim_{t \to 0}\int_{\Omega}\frac{1}{t}\left(\frac{|u+t\varphi|^{p(x)}-|u|^{p(x)}}{p(x)|x|^{\alpha p(x)}}+\frac{\mu(x)|u+t\varphi|^{q(x)}-|u|^{q(x)}}{q(x)|x|^{\alpha q(x)}}\right)dx \nonumber \\
&=\lim_{t \to 0}\int_{\Omega}\frac{1}{t}\left(\frac{d}{d\gamma}\frac{|u+\gamma t\varphi|^{p(x)}}{p(x)|x|^{\alpha p(x)}}\bigg|_{\gamma=\varepsilon_{1}}+\frac{d}{d\gamma}\frac{\mu(x)|u+\gamma t\varphi|^{q(x)}}{q(x)|x|^{\alpha q(x)}}\bigg|_{\gamma=\varepsilon_{2}}\right)dx  \nonumber \\
& =\lim_{t \to 0}\int_{\Omega}\left(|x|^{-\alpha p(x)}|u+\varepsilon_1 t\varphi|^{p(x)-2}(u+\varepsilon_1 t\varphi ) \varphi \right. \nonumber \\
&+ \left. |x|^{-\alpha q(x)}\mu(x)|u+\varepsilon_2 t\varphi|^{q(x)-2}(u+\varepsilon_2 t\varphi) \varphi\right)  dx, \nonumber \\
\end{align}
for all $u, \varphi \in W_0^{1,\mathcal{H}}(\Omega)$, $\gamma \in \mathbb{R}$. Using the Young inequality, we obtain
\begin{align}\label{e3.21f}
  \bigg||x|^{-\alpha p(x)}|u+\varepsilon_1  t\varphi|^{p(x)-2}(u+\varepsilon_1 t\varphi) \varphi \bigg| & \leq k_1 |x|^{-\alpha p(x)}(|u|^{p(x)}+|\varphi|^{p(x)}),
\end{align}
and
\begin{align}\label{e3.21fg}
  \bigg||x|^{-\alpha q(x)}\mu(x)|u+\varepsilon_2 t\varphi|^{q(x)-2}(u+\varepsilon_2 t\varphi) \varphi\bigg| & \leq k_2 |x|^{-\alpha q(x)}(\mu(x)|u|^{q(x)}+\mu(x)|\varphi|^{q(x)}).
\end{align}
where $k_1:=\frac{2^{p(x)-1}(p(x)-1)+1}{p(x)}$, and $k_2:=\frac{2^{q(x)-1}(q(x)-1)+1}{q(x)}$.
\newline Using (\ref{e3.21f}),(\ref{e3.21fg}) and Lemma \ref{Lem:3.3abb}, we can write
\begin{align}\label{e3.21fh}
&\bigg||x|^{-\alpha p(x)}|u+\varepsilon_1  t\varphi|^{p(x)-2}(u+\varepsilon_1 t\varphi) \varphi \bigg|+\bigg||x|^{-\alpha q(x)}\mu(x)|u+\varepsilon_2 t\varphi|^{q(x)-2}(u+\varepsilon_2 t\varphi) \varphi\bigg| \nonumber \\
& \leq k_1(|x|^{-\alpha p(x)}|u|^{p(x)}+|x|^{-\alpha q(x)}\mu(x)|u|^{q(x)})+k_2(|x|^{-\alpha p(x)}|\varphi|^{p(x)}+|x|^{-\alpha q(x)}\mu(x)|\varphi|^{q(x)})\nonumber \\
& \leq k \hat{H}\left(\|u\|_{1,\mathcal{H},0}^{\kappa}+\|\varphi\|_{1,\mathcal{H},0}^{\kappa}\right),
\end{align}
where $k:=\max\{\frac{2^{q^+}(q^+-1)}{p^-},\, k_1,\, k_2\}$. \newline Therefore, by the Lebesgue Dominated Convergence Theorem it reads
\begin{align}\label{e3.21g}
\langle \mathcal{G}^{\prime}(u),\varphi \rangle & =\int_{\Omega}\lim_{t \to 0}\left(|x|^{-\alpha p(x)}|u+\varepsilon_1 t\varphi|^{p(x)-2}(u+\varepsilon_1 t\varphi) \right. \nonumber \\
&+ \left. |x|^{-\alpha q(x)}\mu(x)|u+\varepsilon_2 t\varphi|^{q(x)-2}(u+\varepsilon_2 t\varphi)\right) \varphi  dx, \nonumber \\
&=\int_{\Omega}\left(|x|^{-\alpha p(x)}|u|^{p(x)-2}u+\mu(x)|x|^{-\alpha q(x)}|u|^{q(x)-2}u\right) \varphi dx.
\end{align}
Since the right-hand side of (\ref{e3.21g}), as a function of $\varphi$, is a linear functional, $\mathcal{G}^{\prime}$ is linear on $W_0^{1,\mathcal{H}}(\Omega)$.\\ Next, by the Young inequality and Lemma \ref{Lem:3.3abb}, we obtain
\begin{align}\label{e3.22g}
|\langle \mathcal{G}^{\prime}(u),\varphi \rangle| & \leq \int_{\Omega}\left(|x|^{-p(x)}|u|^{p(x)-1} |\varphi|+|x|^{-\alpha q(x)}\mu(x)|u|^{q(x)-1} |\varphi|\right) dx \nonumber \\
& \leq \frac{q^+}{p^-} \int_{\Omega}\left(\frac{|u|^{p(x)}}{p(x)|x|^{\alpha p(x)}}+\frac{|\varphi|^{p(x)}}{p(x)|x|^{\alpha p(x)}}\right)
+\left(\frac{\mu(x)|u|^{q(x)}}{q(x)|x|^{\alpha q(x)}}+\frac{\mu(x)|\varphi|^{q(x)}}{q(x)|x|^{\alpha q(x)}}\right)dx \nonumber \\
& \leq \frac{q^+}{p^-} \hat{H}\left(\|u\|_{1,\mathcal{H},0}^{\kappa}+\|\varphi\|_{1,\mathcal{H},0}^{\kappa}\right).
\end{align}
Therefore, for all $u \in W_0^{1,\mathcal{H}}(\Omega)$, we have
\begin{align}\label{e3.21gb}
\|\mathcal{G}^{\prime}(u)\|_{W_0^{1,\mathcal{H}}(\Omega)^*}&=\sup_{\|\varphi\|_{1,\mathcal{H},0}\leq 1}|\langle \mathcal{G}^{\prime}(u),\varphi \rangle|
 \leq \frac{q^+}{p^-}\hat{H} \left(\|u\|_{1,\mathcal{H},0}^{\kappa}+1\right),
\end{align}
which means that $\mathcal{G}^{\prime}$ is bounded. Therefore, $\mathcal{G}$ is G\^{a}teaux differentiable whose derivative is given by the formula (\ref{e3.21g}).
Now, we proceed with the continuity of $\mathcal{G}^{\prime}$. Let $(u_n) \subset W_0^{1,\mathcal{H}}(\Omega)$ such that $u_n \to u$ in $W_0^{1,\mathcal{H}}(\Omega)$. Applying the inequality (Proposition 17.2, \cite{Chipot})
\begin{equation}\label{e3.21gbr}
\left|\left\vert \xi\right\vert ^{m-2}\xi-\left\vert \psi\right\vert^{m-2}\psi \right| \leq C_{m}\left\vert\xi-\psi\right\vert\{|\xi|+|\psi|\}^{m-2}, \quad \xi,\psi\in \mathbb{R}^{N},\,\, m> 1,
\end{equation}
gives
\begin{align}\label{e3.21gbd}
\left|\langle\mathcal{G}^{\prime}(u_{n})-\mathcal{G}^{\prime}(u),\varphi\rangle\right|& \leq \int_{\Omega}|x|^{-p(x)}\left||u_{n}|^{p(x)-2}u_{n}-|u_0|^{p(x)-2}u_0 \right||\varphi| dx \nonumber \\
&+ \int_{\Omega}|x|^{-\alpha q(x)}\mu(x)\left||u_{n}|^{q(x)-2}u_{n}-|u|^{q(x)-2}u \right||\varphi| dx \nonumber \\
& \leq C_{p}\int_{\Omega}|x|^{-p(x)}\{|u_{n}|+|u|\}^{p(x)-2}|u_{n}-u| |\varphi|dx \nonumber \\
&+ C_{q}\int_{\Omega}|x|^{-\alpha q(x)}\mu(x)\{|u_{n}|+|u|\}^{q(x)-2}|u_{n}-u||\varphi| dx.
\end{align}
Thus, considering that $u_{n} \to u$ in $L(\Omega)$ and $(u_{n}) $ is bounded, it reads
\begin{equation}\label{e3.21k} \|\mathcal{G}^{\prime}(u_n)-\mathcal{G}^{\prime}(u)\|_{W_0^{1,\mathcal{H}}(\Omega)^*}=\sup_{\|\varphi\|_{1,\mathcal{H},0} \leq1}\left|\langle\mathcal{G}^{\prime}(u_{n})-\mathcal{G}^{\prime}(u),\varphi\rangle\right| \to 0.
\end{equation}
Therefore, $\mathcal{G}$ is of class $C^{1}(W_0^{1,\mathcal{H}}(\Omega),\mathbb{R})$.\\
Finally, $\mathcal{I}^{\prime}(u)=0$ is exactly the weak formulation of \eqref{e1.1}, so critical points of $\mathcal{I}$ are weak solutions. This completes the proof.
\end{proof}

\medskip
\noindent
The main results of the present paper is given below.

\begin{theorem}\label{Thm:3.3}
Assume that the following assumptions hold:\\
\begin{itemize}
  \item [$(\mathbf{\beta_0})$] $\beta \in C_{+}\left(\overline{\Omega }\right)$ such that $q^{+}<\beta^{-}\leq\beta ^{+}<p^{\ast}(x)$\,\, $\forall x\in \overline{\Omega}$.
  \item [$(\mathbf{f}_{1})$] $f:\overline{\Omega }\times\mathbb{R}\rightarrow\mathbb{R}$ is a Carath\'{e}odory function and satisfies the growth condition
\begin{equation*}
\left\vert f(x,t)\right\vert \leq c_{1}+c_{2}\left\vert t\right\vert^{\beta (x)-1},\quad \forall \left( x,t\right) \in \overline{\Omega }\times\mathbb{R},
\end{equation*}
where $c_{1}$ and $c_{2}$ are positive constants.
  \item [$(\mathbf{f}_{2})$] $f(x,t)=o\left( \left\vert t\right\vert^{q^{+}-1}\right)$, $t\rightarrow 0$ uniformly $\forall x\in \overline{\Omega}$.
  \item [$(\mathbf{AR})$] $\exists K>0$, $\theta >q^{+}$ such that
  \begin{equation*}
  0<\theta F(x,t)\leq f(x,t)t,\quad \left\vert t\right\vert \geq K \text{ a.e.}\, x\in \overline{\Omega }.
  \end{equation*}
\end{itemize}
Then problem \eqref{e1.1} has at least one nontrivial positive weak solution.
\end{theorem}

\medskip
\noindent
Note that $(\mathbf{AR})$ stands for the Ambrosetti–Rabinowitz condition \cite{Ambrosetti-Rabinowitz}.\\

\medskip
\noindent
To obtain the result of Theorem \ref{Thm:3.3}, we need to show Lemma  \ref{Lem:3.4} and Lemma \ref{Lem:3.5} hold.

\begin{lemma}\label{Lem:3.4}
Suppose  $(\mathbf{\beta_0})$, $(\mathbf{f}_{1})$, $(\mathbf{f}_{2})$ and $(\mathbf{AR})$
hold. Then the following statements hold:
\begin{itemize}
  \item [$(i)$] There exist two positive real numbers $\gamma $ and $\eta$ such that $\mathcal{I}(u)\geq \eta>0$, for all $u\in W_0^{1,\mathcal{H}}(\Omega)$
with $\|u\|_{1,\mathcal{H},0} =\gamma $.
  \item [$(ii)$] There exists a $\hat{u}\in W_0^{1,\mathcal{H}}(\Omega)$ such that $\|\hat{u}\|_{1,\mathcal{H},0} >\gamma $\, and \,  $\mathcal{I}(\hat{u})<0$.
\end{itemize}
\end{lemma}
\begin{proof}
$(i)$ By $(\mathbf{f}_{1})$ and $(\mathbf{f}_{2})$, one can write
\begin{equation}\label{e4.a1}
F(x,t) \leq \varepsilon |t|^{q^+} + C_{\varepsilon} |t|^{\beta(x)}, \quad \forall \left( x,t\right) \in \overline{\Omega }\times\mathbb{R}.
\end{equation}
Then
\begin{align}\label{e4.b1}
\mathcal{I}(u)&\geq\frac{1}{q^+}\rho_\mathcal{H}(\nabla u)+\frac{1}{q^+M^{\alpha\tau}}\rho_\mathcal{H}(u)-\lambda\varepsilon\int_{\Omega}|u|^{q^+}dx-\lambda C_{\varepsilon}\int_{\Omega}|u|^{\beta(x)}dx \nonumber\\
&\geq \frac{1}{q^+}\rho_\mathcal{H}(\nabla u)-\lambda\varepsilon|u|^{q^+}_{q^+}-\lambda C_{\varepsilon}\left(|u|^{\beta^+}_{\beta^+}+|u|^{\beta^-}_{\beta^-} \right).
\end{align}
Using Proposition \ref{Prop:2.7a}, and assuming that $\|u\|_{1,\mathcal{H},0} <1$ provides
\begin{align}\label{e4.c1}
\mathcal{I}(u)&\geq \frac{1}{q^+}\|u\|_{1,\mathcal{H},0}^{q^+}-\lambda\varepsilon c_3\|u\|_{1,\mathcal{H},0}^{q^+}-\lambda C_{\varepsilon} c_4\|u\|_{1,\mathcal{H},0}^{\beta^-} \nonumber\\
&=\left(\frac{1}{q^+}-\lambda\varepsilon c_3 \right)\|u\|_{1,\mathcal{H},0}^{q^+}-\lambda C_{\varepsilon} c_4 \|u\|_{1,\mathcal{H},0}^{\beta^{-}}.
\end{align}
Then if we choose an upper bound $\hat{\lambda}$ as $\hat{\lambda}=\frac{1}{q^+C(\varepsilon)}$, then for any $\lambda \in (0,\hat{\lambda})$ and $\gamma=\|u\|_{1,\mathcal{H},0} <1$ small enough, there exists a real number $\eta$ such that $\mathcal{I}(u)\geq \eta>0$  for any $u\in W_0^{1,\mathcal{H}}(\Omega)$.\\
$(ii)$ Due to $(\mathbf{AR})$, we can write
\begin{equation}\label{e4.e1}
F(x,t) \geq  c_5|t|^{\theta}, \quad \forall x \in \overline{\Omega },\,\, |t| \geq K.
\end{equation}
Thus, for $0\neq\phi\in W_0^{1,\mathcal{H}}(\Omega)$ and $t>1$, we have
\begin{align}\label{e4.f1}
\mathcal{I}(t\phi)&\leq\frac{t^{q^{+}}}{p^-}\left(\rho_\mathcal{H}(\nabla\phi)+\int_{\Omega}\left(\frac{|\phi|^{p(x)}}{|x|^{\alpha p(x)}}+\mu(x)\frac{|\phi|^{q(x)}}{|x|^{\alpha q(x)}}\right)dx\right)-\lambda c_5 t^{\theta}\int_{\Omega}|u|^{\theta}dx,
\end{align}
which implies that $\mathcal{I}(t\phi) \to -\infty$ as $t \to \infty$. Therefore, letting $\hat{u}=t\phi$, and hence $\|\hat{u}\|_{1,\mathcal{H},0}=t\|\phi\|_{1,\mathcal{H},0}$, leads to  $\|\hat{u}\|_{1,\mathcal{H},0} >1 >\gamma$ and  $\mathcal{I}(\hat{u})<0$ provided $t$ is large enough.
\end{proof}

\begin{definition}\label{Def:2.1}
Let $X$ be a Banach space, and $J:X\rightarrow\mathbb{R}$ be a $C^{1}$-functional. We say that $J$ satisfies the Palais-Smale condition ($(\mathbf{PS})$ for short) \cite{Ambrosetti-Rabinowitz} if: every
sequence $(u_{n}) \subset X$ such that $J(u_{n})$ is bounded and $J^{\prime }(u_{n}) \rightarrow 0$ in $X^*$ admits a convergent subsequence in $X$.
\end{definition}

\begin{lemma}\label{Lem:3.5}
Suppose $(\mathbf{\beta_0})$, $(\mathbf{f}_{1})$ and $(\mathbf{AR})$ hold. Then $\mathcal{I}$ satisfies the $(\mathbf{PS})$
condition.
\end{lemma}
\begin{proof}
Let assume that there exists a sequence $(u_{n}) \subset W_0^{1,\mathcal{H}}(\Omega)$ such that
\begin{equation}\label{e4.1}
|\mathcal{I}(u_{n})| \leq c_0 \quad \text{ and } \quad \|\mathcal{I}^{\prime}(u_{n})\|_{W_0^{1,\mathcal{H}}(\Omega)^*}\rightarrow 0,\quad 0<c_0 \in \mathbb{R}.
\end{equation}
Using \eqref{e4.1}, $(\mathbf{AR})$, and Proposition \ref{Prop:2.2a}, we obtain
\begin{align}\label{e4.2}
c_0&\geq \mathcal{I}(u_n)=\varrho_{\mathcal{H}}(u_n)+\mathcal{G}(u_n)-\lambda\int_{\Omega}F(x,u_n)dx\nonumber\\
& \geq \left(\frac{1}{q^+}-\frac{1}{\theta}\right)\rho_\mathcal{H}(\nabla u_n)+\frac{1}{\theta}\int_{\Omega}\left(\rho_\mathcal{H}(\nabla u_n)+\frac{| u_n|^{p(x)}}{|x|^{\alpha p(x)}}+\mu(x)\frac{| u_n|^{q(x)}}{|x|^{\alpha q(x)}}-\lambda f(x,u_n)u_n\right)dx\nonumber\\
& \geq \left(\frac{1}{q^+}-\frac{1}{\theta}\right) \|u_{n}\|_{1,\mathcal{H},0}^{p^-}-\frac{1}{\theta}\|u_{n}\|_{1,\mathcal{H},0}\|\mathcal{I}^{\prime}(u_{n})\|_{W_0^{1,\mathcal{H}}(\Omega)^*}.
\end{align}
Therefore $(u_{n}) $ is bounded in $W_0^{1,\mathcal{H}}(\Omega)$. Since $W_0^{1,\mathcal{H}}(\Omega)$ is reflexive, passing to a subsequence, we have $u_{n}\rightharpoonup u_0 \in W_0^{1,\mathcal{H}}(\Omega)$. Then using $(\mathbf{f}_{1})$, Hölder inequality, and Proposition \ref{Prop:2.7a},  we obtain
\begin{align}\label{e4.3}
\left|\langle f(x,u_{n}), u_{n}-u_0\rangle\right|\leq c_1\left||u_n|^{\beta(x)-1} \right|_{\frac{\beta(x)}{\beta(x)-1}}\left| u_{n}-u_0\right|_{\beta(x)}+c_2\int_{\Omega}\left| u_{n}-u_0\right|dx \to 0.
\end{align}
First let's write
\begin{align}\label{e4.3a}
\langle\mathcal{G}^{\prime}(u_{n}),u_{n}-u_0\rangle= \langle\mathcal{G}^{\prime}(u_{n})-\mathcal{G}^{\prime}(u_{0}),u_{n}-u_0\rangle+\langle\mathcal{G}^{\prime}(u_{0}),u_{n}-u_0\rangle.
\end{align}
Taking into consideration that $\mathcal{G}^{\prime}(u_{0}) \in W^{1,\mathcal{H}}(\Omega)^*$, we obtain
\begin{align}\label{e4.4}
\langle\mathcal{G}^{\prime}(u_{0}),u_{n}-u_0\rangle=\int_{\Omega}\mathcal{G}^{\prime}(u_{0}) (u_{n}-u_0)dx \to 0.
\end{align}
Moreover, since $u_{n} \to u_0$ in $L^2(\Omega)$, $(u_{n})$ is bounded. Thus
\begin{align}\label{e4.4a}
\langle\mathcal{G}^{\prime}(u_{n})-\mathcal{G}^{\prime}(u_{0}),u_{n}-u_0\rangle&= \int_{\Omega}|x|^{-\alpha p(x)}\left(|u_{n}|^{p(x)-2}u_{n}-|u_0|^{p(x)-2}u_0 \right)(u_{n}-u_0) dx \nonumber \\
&+ \int_{\Omega}|x|^{-\alpha q(x)}\mu(x)\left(|u_{n}|^{q(x)-2}u_{n}-|u_0|^{q(x)-2}u_0 \right)(u_{n}-u_0) dx \nonumber \\
& \leq \int_{\Omega}|x|^{-\alpha p(x)}\left||u_{n}|^{p(x)-2}u_{n}-|u_0|^{p(x)-2}u_0 \right||u_{n}-u_0| dx \nonumber \\
&+ \int_{\Omega}|x|^{-\alpha q(x)}\mu(x)\left||u_{n}|^{q(x)-2}u_{n}-|u_0|^{q(x)-2}u_0 \right||u_{n}-u_0| dx \nonumber \\
& \leq C_{p}\int_{\Omega}|x|^{-\alpha p(x)}\{|u_{n}|+|u_{0}|\}^{p(x)-2}|u_{n}-u_0|^2 dx \nonumber \\
&+ C_{q}\int_{\Omega}|x|^{-\alpha q(x)}\mu(x)\{|u_{n}|+|u_{0}|\}^{q(x)-2}|u_{n}-u_0|^2 dx \to 0
\end{align}
where we applied the inequality \eqref{e3.21gbr}.
Thus
\begin{align}\label{e4.4b}
\langle\mathcal{G}^{\prime}(u_{n}),u_{n}-u_0\rangle \to 0.
\end{align}
Now, using \eqref{e4.1} gives
\begin{align}\label{e4.5}
\langle \mathcal{I}^\prime (u_{n}), u_{n}-u_0\rangle&=\langle\varrho^{\prime}_{\mathcal{H}}(u_{n}),u_{n}-u_0\rangle+\langle\mathcal{G}^{\prime}(u_{n}),u_{n}-u_0\rangle-\lambda\langle f(x,u_{n}), u_{n}-u_0\rangle \to 0.
\end{align}
However, \eqref{e4.3}, \eqref{e4.4b} and \eqref{e4.5} together implies that
\begin{align}\label{e4.6}
\langle\varrho^{\prime}_{\mathcal{H}}(u_{n}),u_{n}-u_0\rangle \to 0.
\end{align}
Thus, by Proposition \ref{Prop:2.7} (($S_+$)-property), $u_{n} \to u_0 \in W_0^{1,\mathcal{H}}(\Omega)$. In conclusion,  $\mathcal{I}$ satisfies the $(\mathbf{PS})$
condition.
\end{proof}

\begin{proof}[Proof of Theorem \ref{Thm:3.3}] By Lemma \ref{Lem:3.4}, $\mathcal{I}$ has Mountain-Pass geometry; and by Lemma \ref{Lem:3.5}, $\mathcal{I}$ satisfies $(\mathbf{PS})$. Also, $\mathcal{I}(0) =0$. Thus, by the \textit{Mountain-Pass Theorem}  \cite{Willem}:

\begin{itemize}
\item $u_0$ is a critical point of $\mathcal{I}$ at mountain pass level $c_{MP}$:
    \[
      c_{MP} := \inf_{\gamma \in \Gamma} \max_{t \in [0,1]}\mathcal{ I}(\gamma(t)) \geq \eta > 0,
    \]
    where $\Gamma = \{\gamma \in C([0,1], W^{1,\mathcal{H}}_0(\Omega)) : \gamma(0) = 0, \gamma(1) = \hat{u}\}$.
    \item Since $\mathcal{I}(u_0) \geq \eta > 0$, $u_0 \neq 0$ is a nontrivial critical point of $\mathcal{I}$ at level $c_{MP}$;\\ $\mathcal{I}(u_0)=c_{MP}$ and $\mathcal{I}'(u_0)=0$.
\end{itemize}
Therefore, by Lemma \ref{Lem:3.1}, $u_0 \in W_0^{1,\mathcal{H}}(\Omega)$ corresponds to a nontrivial weak solution of \eqref{e1.1}.\\

\medskip

\noindent
Define $f_{+}:\overline{\Omega }\times\mathbb{R}\rightarrow [0,\infty)$ by \\
$ f_{+}( x,t) =\left\{\begin{array}{cc}
f(x,t) & \text{if }\, t\geq0, \\
0 & \text{if }\, t<0.
\end{array}
\right.$\\
Then clearly the modified energy functional $\mathcal{I_{+}}$ is also of class $C^1$ on $W_0^{1,\mathcal{H}}(\Omega)$.\\
\medskip

\noindent
\textbf{Step 1.} \\
Define $u^-:=\max\{-u,0 \}$. We test \eqref{e3.2} by replacing  $u^-$ with $\varphi$ which  gives
\begin{align}\label{e3.2ac}
&\int_{\Omega}(|\nabla u|^{p(x)-2}\nabla u+\mu(x)|\nabla u|^{q(x)-2}\nabla u)\cdot\nabla u^- dx+\int_{\Omega}\left(\frac{|u|^{p(x)-2}u}{|x|^{\alpha p(x)}}+\mu(x)\frac{|u|^{q(x)-2}u}{|x|^{\alpha q(x)}}\right)u^- dx\nonumber\\
&=\lambda\int_{\Omega}f(x,u) u^- dx
\end{align}
for a nontrivial weak solution $u$ of \eqref{e1.1}. Then considering the decomposition $\Omega=\{x \in \Omega: u(x)\geq 0\} \cup \{x \in \Omega: u(x)<0\}$ for \eqref{e3.2ac}, we have
\begin{align}\label{e3.2ad}
&0\geq -\int_{\{u<0\}}(|\nabla u|^{p(x)}+\mu(x)|\nabla u|^{q(x)}) dx-\int_{\{u<0\}}\left(\frac{|u|^{p(x)}}{|x|^{\alpha p(x)}}+\mu(x)\frac{|u|^{q(x)}}{|x|^{\alpha q(x)}}\right) dx\nonumber\\
&=\lambda\int_{\{u\geq0\}}f(x,u) u^- dx \geq 0,
\end{align}
which implies that $u^-\equiv 0$. So, since $u=0$ on $\partial \Omega$, by the weak minimum principle there is a set of positive measure where $u \geq 0$, and $u \not\equiv 0$.\\
\medskip

\noindent
\textbf{Step 2.} \\
We claim that $u(x) > 0$ for all $x\in \Omega$. Suppose, on the contrary, that there exists an interior point $x_0 \in \Omega$ such that $u(x_0)=0$. Since  $u \geq 0$ a.e., this must be a global minimum of $u$. This, by the strong minimum principle, implies that $u$ is constant. That is, $u\equiv 0$ in a neighborhood of $x_0$. However, this contradicts $u \not\equiv 0$. Therefore $u(x) > 0$ for all $x\in \Omega$.

\end{proof}

\suppress{
\section*{Appendix}

\begin{lemma}\label{Lem:3.1a}
$\mathcal{G}^{\prime}$ is strictly monotone.
\end{lemma}
\begin{proof}
  For $u, v \in W_0^{1,\mathcal{H}}(\Omega)$ with $u\neq v$, we have
\begin{align}\label{e3.11zz}
  \langle \mathcal{G}^{\prime}(u)-\mathcal{G}^{\prime}(v), u- v \rangle &= \int_{\Omega}|x|^{-\alpha p(x)}\left(|u|^{p(x)-2}u-|v|^{p(x)-2}\upsilon \right)(u-v) dx \nonumber \\
  &+ \int_{\Omega}|x|^{-\alpha q(x)}\mu(x)\left(|u|^{q(x)-2}u-|v|^{q(x)-2}\upsilon \right)(u-v) dx \nonumber \\
  &\geq 2^{-q^{+}}M^{-\tau}\left(\int_{\Omega}|u-v|^{p(x)}dx +\int_{\Omega}\mu(x)|u-v|^{q(x)}dx\right) > 0,
\end{align}
where we apply the well-known inequality
\begin{equation*}
\left( \left\vert \xi\right\vert ^{s-2}\xi-\left\vert \psi\right\vert^{s-2}\psi \right)\cdot\left(\xi-\psi\right) \geq 2^{-s}\left\vert\xi-\psi\right\vert^{s}; \quad \xi,\psi\in \mathbb{R}^{N},\,\, s> 1.
\end{equation*}
\end{proof}
}

\section*{Funding}
This work was supported by Athabasca University Research Incentive Account [140111 RIA].

\section*{ORCID}
https://orcid.org/0000-0002-6001-627X

\end{document}